\documentclass[noinfoline]{imsart}

\RequirePackage[OT1]{fontenc}
\RequirePackage{amsthm,amsmath,natbib}
\RequirePackage[colorlinks,citecolor=blue,urlcolor=blue]{hyperref}

\RequirePackage{hypernat}
\RequirePackage{multirow}
\RequirePackage{verbatim}
\RequirePackage{epsfig}
\RequirePackage{subfigure}
\RequirePackage{breakurl}
\RequirePackage{breakcites}
\RequirePackage{amssymb}

\startlocaldefs

\newcommand{\ds} {\displaystyle}
\theoremstyle{plain}
\newtheorem{theorem}{Theorem}[section]
\newtheorem{lemma}[theorem]{Lemma}
\newtheorem{proposition}[theorem]{Proposition}
\newtheorem{conjecture}[theorem]{Conjecture}
\newtheorem{corollary}[theorem]{Corollary}
\newtheorem{problem}[theorem]{Problem}
\newenvironment{definition}[1][Definition]{\begin{trivlist}
\item[\hskip \labelsep {\bfseries #1}]}{\end{trivlist}}
    \newcounter{example}
    \newenvironment{example}[1][]{\refstepcounter{example}\par\medskip\noindent%
       \textbf{Example~\theexample. #1} \rmfamily}{\medskip}
\newtheorem{remark}{Remark}[section]  

\newcommand{\beq}{\begin{equation}}
\newcommand{\eeq}{\end{equation}}
\newcommand{\ben}{\begin{eqnarray}}
\newcommand{\een}{\end{eqnarray}}
\newcommand{\beno}{\begin{eqnarray*}}
\newcommand{\eeno}{\end{eqnarray*}}

\endlocaldefs
\begin{document}
\begin{frontmatter}

\title{Fibers of  multi-way contingency tables given conditionals: relation to marginals, cell bounds and Markov bases}
\runtitle{Conditional and Marginal Sample Spaces}


\author{\fnms{Aleksandra B.} \snm{Slavkovi{\'c}}\thanksref{t2}\corref{}\ead[label=e1]{sesa@stat.psu.edu}}
\address{\printead{e1}}
\and
\author{\fnms{Xiaotian} \snm{Zhu}\thanksref{t2}\ead[label=e2]{xxz131@psu.edu}}
\address{\printead{e2}}
\and
\author{\fnms{Sonja} \snm{Petrovi{\'c}}\thanksref{t3}\ead[label=e3]{Sonja.Petrovic@iit.edu}}
\address{\printead{e3}}
\runauthor{Slavkovi{\'c}}

\runauthor{Slavkovi{\'c} et al.}

\thankstext{t2}{Supported in part by NSF grants SES-052407 and BCS-0941553 to the Pennsylvania State University.}
\thankstext{t3}{Supported in part by grant FA9550-12-1-0392 from the U.S. Air Force Office of Scientific Research (AFOSR) and the Defense Advanced Research Projects Agency (DARPA)}

\begin{abstract}
A reference set, or a fiber, of a contingency table is the space of all realizations of the table under a given set of constraints such as marginal totals. Understanding the geometry of this space is a key problem in algebraic statistics, important for conducting exact conditional inference, calculating cell bounds, imputing missing cell values, and assessing the risk of disclosure of sensitive information. 

Motivated primarily by disclosure limitation problems where constraints can come from summary statistics other than the margins, in this paper 
we study the space $\mathcal{F_T}$ of all possible multi-way contingency tables for a given sample size and set of observed conditional frequencies. We show that this space can be decomposed according to different possible marginals, which, in turn, are encoded by the solution set of a linear Diophantine equation. 
We characterize the difference between two fibers: $\mathcal{F_T}$ and 
the space of tables for a given set of corresponding marginal totals. In particular, we solve a generalization of an open problem posed by \cite{dobra2008asa}. 
Our decomposition of $\mathcal{F_T}$ has two important consequences: (1) we derive new cell bounds, some including connections to Directed Acyclic Graphs, and (2) we describe a structure for the Markov bases for the space $\mathcal{F_T}$ that leads to a simplified calculation of Markov bases in this particular setting. \\
\end{abstract}



\end{frontmatter}

\section{Introduction}
\label{sec:introduction}
\label{sec:intro}

In \cite{dobra2008asa}, the authors use tools from algebraic statistics to study two related problems: maximum likelihood estimation for log-linear models in multi-way contingency tables, and disclosure limitation strategies to protect against the identification of individuals associated with small counts in the tables; for an overview of disclosure limitation literature see \cite{doyle-confidentiality} and \cite{ hundepool2012sdc}. These are linked to the general problem of inference in tables for which only partial information is available (e.g., see \cite{dobra2006dam}, \cite{yves2003}, and \cite{marjoram2003markov}).

Incomplete data commonly arise in surveys or census data which have been modified to limit disclosure of sensitive information. 
Instead of releasing complete data, summary statistics are often released, even if they may not be the sufficient statistics for the probability model. Examples of summary statistics are marginal tables, or tables of conditional frequencies, e.g., \cite{slavkovic09}.   Given a set of released statistics, there are a number of ways to assess the disclosure risk and data utility, 
including  computing bounds for cell entries,
 enumerating all table realizations, 
and  sampling from a fiber to estimate posterior distributions. A \emph{fiber} is the space of all possible tables consistent with the observed statistics. Since the fibers form the support of the conditional distributions given a set of summary statistics, their properties are important for conducting exact conditional inference; e.g., see \cite{diaconis1998}  for an algebraic statistics approach to goodness-of-fit testing given the marginal totals, and \cite{dobra:fien:2009}) for calculating bounds on the cell entries.
Similar techniques that rely on understanding fibers' structure can be used to impute missing data in contingency tables and to create replacement tables; see \cite{slav:lee:2009}, with focus on tables that arise from preserving conditional frequencies. 
In this paper, we study the sample space of contingency tables given observed conditional frequencies and their relations to corresponding marginals. More specifically, we address the following challenge:

\begin{problem}[Problem 5.7 in \cite{dobra2008asa}] 
\label{mainProb}
Characterize the difference of two fibers, one for a conditional probability array, and the other for the corresponding margin, and thus simplify the calculation of Markov bases for the conditionals by using the knowledge of the moves of the corresponding margins.
\end{problem}

\smallskip
Here is a general setup. Consider $r$ categorical random variables, $X_1, \ldots, X_r$, where each $X_i$ takes values in the finite set of categories $[d_i] \equiv \{ 1, \ldots, d_i \}$. Let  $\mathcal{D} = \bigotimes_{i=1}^r [d_i]$, and $\mathbb{R}^\mathcal{D}$ be the vector space of $r$-dimensional arrays of format $d_1 \times \ldots \times d_r$, with a total of $d = \prod_i d_i$ entries.
The cross-classification of $n$ independent and identically distributed realizations of  $(X_1, \ldots, X_r)$ produces a random integer-valued array ${\bf n} \in \mathbb{R}^\mathcal{D}$, called a $r$-way \textit{contingency table}, whose coordinate entry
$
n_{i_i, \ldots, i_r}
$
is the number of times the label combination, or \textit{cell}, $(i_1, \ldots, i_r)$ is observed in the sample (see \cite{agresti2002, bishop2007discrete,Lauritzen1996} for details). It is often convenient to order the cells in some prespecified way (e.g., lexicographically). 

Let $A$ and $B$ be proper subsets of $\{X_1, X_2, ... , X_r\}$, and $C=\{X_1, X_2, ... , X_r\}\setminus (A\cup B)$. 
We can regard $A, B$ and $C$ as three categorical variables with levels $A_1, ..., A_I, B_1, ...., B_J, $ and $C_1, ..., C_K.$ Thus, we can summarize the $r$-way table $\bf{n}$ as a $3$-way table $\bf{n^*}:=\{s_{ijk}\}$, where $s_{ijk}$ is the count in the cell  $(A_i, B_j, C_k)$.  Finally, let $c_{ij}$ be the observed conditional frequency $P(A=i|B=j),$ such that $\sum_i P(A=i|B=j)=1$. 
 If $C$ is an empty set, we refer to $c_{ij}$'s as {\em full} conditionals, otherwise as {\em small} or {\em partial} conditionals.

Motivated by Problem~\ref{mainProb}, we investigate the fiber  $\mathcal{F_T}$ for $\mathcal{T}=\{P(A|B), N\}$, that is the 
space of all possible tables consistent with:
\begin{itemize}
    \item[(a)] the observed grand total, $\sum\limits_{i_1...i_r}{n_{i_1i_2...i_r}}=N$, and
    \item[(b)] a set of observed conditional frequencies, $P(A|B)$. 
\end{itemize}
 Note that we do not observe the values of $B$, and we assume that all of the given frequencies are exact. 
Then,  the space $\mathcal{F_T}$ is the set of integer  solutions to the following system of linear equations

\beq\label{eq:spaceoftables} \left\{ \begin{array}{l}
 M\bf{n} = t \\
 \text{every B marginal} > 0 \\
 \end{array} \right\},
\eeq where  $\bf{n}$ and $\bf{t}$ are length $d$ column vectors, and $M$ is a $(J+1)\times d$ matrix that, together with $\bf{t}$, describes the information encoded by the grand
total and the given frequencies. When $N$ is clear from the context, we use the shorthand notation $\mathcal{F}_{A|B}$ to denote $\mathcal{F}_{\{P(A|B),N
\}}$.
The space of tables given the  $[AB]$ marginal counts $s_{ij+}$ is denoted by $\mathcal F_{AB}.$ For a concrete example, see Section~\ref{subsec:mp3}

The main contributions of this manuscript come from the structural results for the fibers defined above.  In particular, we solve a generalization of an open problem posed by \cite{dobra2008asa}. In Corollary \ref{cor:equalspace} we give conditions for when the two fibers $\mathcal{F}_{A|B}$ and $\mathcal{F}_{AB}$ agree.  A decomposition of the table space $\mathcal{F_{A|B}}$ is given in Corollary \ref{cor:tableSpaceDecompositionByMarginals}, showing that the space of tables given the conditional is a disjoint union of spaces of tables given distinct marginals. This decomposition of $\mathcal{F_T}$ leads to three important applied results: (1) in Section~\ref{subsec:countingTables}, we derive new results on computing the exact and approximate cardinality of the given fibers and provide functions to do this in R, (2) in Section~\ref{subsec:bounds},  we derive new cell bounds, some including connections to Directed Acyclic Graphs in Section~\ref{subsec:dag}, and (3) in Section~\ref{subsec:MB}, we describe a structure for the Markov bases for the space $\mathcal{F_T}$ that leads to a simplified calculation of Markov bases in this setting. In Section~\ref{sec:examples}, we demonstrate our theoretical results with a series of simple examples and conclude with a brief discussion in Section~\ref{sec:conclusions}.

\section{The Space of Tables with Given Conditional Frequencies}
\label{sec:mathOfTableSpace}

Data examples suggest a connection between the solutions to a Diophantine equation defined below in equation (\ref{eq:diophantine}), and the
space of tables $\mathcal{F}_{A|B}$ that we are interested in. 
Moreover, this connection appears in symbolic computation: points in the fiber are lattice points in polytopes, and their connection to Diophantine equations has a history in mathematics  \cite{DeLoeraHemmeckeYoshida}.
In what follows, we establish this connection more rigorously from the point of view of marginal and conditional tables.  Finding solutions to Diophantine equations is a well-studied classical
problem in mathematics, one that is generally hard to solve and with a number of proposed algorithms; e.g., see \cite{MoritoSalkin:1980, ChenLi, EisTeWi, Smarandache2000}, and references therein. 
But the equation (\ref{eq:diophantine}) here is simple enough that can be analyzed using classical algebra, and as such affords implementations of simple functions in R needed for statistical analyses. Throughout, we use the notation  established in Section~\ref{sec:intro}. 

\subsection{Table space decomposition}
\label{subsec:tableSpaceAndDioph}

The table of observed conditional frequencies gives rise to a linear Diophantine equation  (\ref{eq:diophantine})  whose solutions  correspond to  possible marginals $B$ that we condition on in $P(A|B).$ Once we know the corresponding marginals $AB,$ we can decompose the table space $\mathcal{F_{A|B}}$ accordingly. 

	The observed conditional frequencies $c_{ij}$ can be used to recover marginal values $s_{+j+}$ in the following way. 
\begin{theorem}
    \label{thm:diophantineSolutions=Marginals}
     Suppose $c_{ij}=\ds\frac{g_{ij}}{h_{ij}}$ for nonnegative and relatively prime integers  $g_{ij}$ and $h_{ij}.$
    Let $m_j$ be the least common multiple of all $h_{ij}$ for fixed $j$. 
    Then, each positive integer solution $\{x_j\}_{j=1}^{J}$ of
    \begin{align}\label{eq:diophantine}
        \sum\limits_{j=1}^{J}m_{j}\cdot x_{j}=N
    \end{align}
    corresponds to a marginal $s_{+j+}$, up to a scalar multiple.  In particular, a table $\bf{n}$ consistent with the given information $\{c_{ij},N\}$ exists
    if and only if Equation (\ref{eq:diophantine}) has a nonnegative integer solution. \end{theorem}

\begin{remark} \rm
    If we allow the solutions to be only integers, then an equation of the form (\ref{eq:diophantine}) is called
    a \emph{linear Diophantine equation}.
\end{remark}

The proof of the above Theorem can be found in Appendix A (Section~\ref{sec:app-proofs}).
Since each solution of the Diophantine equation corresponds to a marginal we condition on, we easily obtain the following consequence:

\begin{corollary}\label{cor:equalspace}
	The following statements are equivalent:
    \begin{itemize}
        \item[(a)] $\mathcal F_{A|B}$ coincides with $\mathcal F_{AB}$.
        \item[(b)] Equation  (\ref{eq:diophantine}) has only one positive integer solution.
    \end{itemize}
\end{corollary}

Note that the tables in these fibers form the support of the conditional distributions given some summary statistics. In the case of margins, there has been much work on conditional exact inference given the marginals as sufficient statistics. Also note that a marginal determines the exact (integer) cell bounds of $\bf{n}$:
the cell bound for $n_{i_1i_2...i_r}$ 
is $[0,s_{+j+}\cdot c_{ij}]$, and a  different marginal $\{s_{+j+}\} $ leads to a different cell
bound. When Corollary~\ref{cor:equalspace} holds, there is only one $AB$ margin.  Thus, the support of conditional distribution given $\{A|B, N\}$ is the same as the support given $AB$ and the integer cell bounds are the same, i.e., $0\leq s_{ijk} \leq s_{ij+}$, that is, $0\leq n_{i_1, ..., i_r} \leq n_{ab}$ in the corresponding $r$-way table.

Let us single out another very important consequence of Theorem \ref{thm:diophantineSolutions=Marginals}, which we will refer to as
the table-space decomposition result:
\begin{corollary} [Table-Space $\mathcal{F_{A|B}}$ Decomposition] 
    \label{cor:tableSpaceDecompositionByMarginals}
Suppose that the Diophantine equation (\ref{eq:diophantine}) has $m$ solutions.
    Denote by $\mathfrak p_i$ the marginal corresponding to the $i^{th}$ solution, and by $\mathcal F_{AB}(\mathfrak p_i)$  the    space of tables given that particular marginal table. Then, we have the following decomposition of the table space, taken as a disjoint union:
    \begin{align*}
            \mathcal F_{A|B} = \bigcup_{i=1}^{m} \mathcal F_{AB}(\mathfrak p_i).
    \end{align*}
  
\end{corollary}
To conclude this section, note that the proof of Theorem \ref{thm:diophantineSolutions=Marginals} shows that each solution $(x_1,\dots,x_J)$ to the Diophantine equation (\ref{eq:diophantine}) corresponds to a marginal in the following way: $s_{+j+}=m_jx_j$ for $1\leq j\leq J$; thus, $s_{ij+}=m_jx_jc_{ij}$. We will use this fact often.

\subsection{The space of tables and integer points in polyhedra}
\label{subsec:tableSpaceAndIntegerPts}

An important question arises next: How many marginals can there be for a given conditional table? This question can be answered using a straightforward count of lattice points in a polyhedron. 
Counting lattice points in polyhedra and counting the solutions to a Diophantine equation (e.g.,  \cite{Sertoz1998} and \cite{ChenLi}\footnote{We note that our Diophantine equation does not necessarily satisfy the main hypothesis of the main result from \cite{ChenLi}.}) are interesting mathematical problems with a rich history.  In particular,  there exist polynomial time algorithms for counting the number of lattice points in polyhedra; e.g., see \cite{Barvinok} and  \cite{Lasserre}.  Due to the simpler geometry of our problem, we do not need to use the general algorithms, and, therefore, we derive simpler solutions. 

 We explain the  correspondence between solutions of Equation (\ref{eq:diophantine}) and nonnegative lattice points $\mathfrak p_i$.

\begin{lemma}
\label{lm:structureOfSolnsOfDioph}
    Suppose that the Diophantine equation (\ref{eq:diophantine}) has a    solution ${\bf x}_0$.
    Then there exist vectors ${\bf v}_1,\dots,{\bf v_{J-1}}\in\mathbb Z^J$ such that \emph{any}
    solution ${\bf x}=(x_1,x_2,...,x_J)$ of (\ref{eq:diophantine}) is given as their integral linear combination:
    $${\bf x}={\bf x}_0+\sum\limits_{i=1}^{J-1}{q_i\cdot {\bf v}_{i}}.$$
    Note that we require that each $q_i\in\mathbb Z$, and that ${\bf v}_1,\dots,{\bf v_{J-1}}$ can be computed from the Diophantine coefficients $m_j$. 
\end{lemma}
The proof of this result uses elementary algebra (and some number theory).  For reader's convenience, it is included in Appendix \ref{sec:proof-of-thm:structure}. 
For additional details and a low-dimensional example illustrating this lemma, see Appendix \ref{sec:app-code}. 

That the set of all solutions to equation (\ref{eq:diophantine}) is a $(J-1)$-dimensional lattice  is a special case of a classical result that identifies the solution set of any system of linear
 Diophantine equations with a lattice \cite{FelixLazebnik}. 
As a subset of that lattice, the set of \emph{nonnegative} solutions can be expressed as a linear combination of the elements in some basis of the lattice. 
In the proof of Lemma~\ref{lm:structureOfSolnsOfDioph}, we give one such combination.
We use this construction to write a {\em solvequick()} function in R (see Appendix \ref{sec:app-code}) for quickly finding a solution to (\ref{eq:diophantine}), and demonstrate its use in Section~\ref{sec:examples}. When there is more than one solution, we provide a quick way to count the tables via a {\em tablecount()} function as explained next.

\subsection{Size of table space: exact and approximate}
    \label{subsec:countingTables}

First we derive the exact count formula for the total number of integer-valued $r$-way tables ${\bf n}$
given the marginal $[AB]$.  In Corollary \ref{cor:exactTotalTableCount}, this count is combined with the table-space decomposition results from Corollary~\ref{cor:tableSpaceDecompositionByMarginals} to derive the number of $r$-way tables in the fiber $\mathcal F_{A|B}$.

Consider a $r$-way table as a $3$-way table of counts $s_{ijk}$ for $A$, $B$, and $C$ taking $I, J,$ and $K$ states, respectively.
Suppose we marginalize $C$.
One can derive a simple formula for the number of $3$-way tables, and, therefore, corresponding $r$-way tables, all having the same margin $[AB]$.

\begin{lemma}[Exact count of data tables given one marginal]
    \label{lm:tableCountFormula}
    Adopting the above notation,  the number of $r$-way tables (data tables) given one marginal $[AB]$ equals 
    \begin{equation}\label{eq:binom1coef}
        |\mathcal F_{AB}| = \prod_{1\leq i\leq I, 1\leq j\leq J} {{s_{ij+} + K-1} \choose {K-1}}.
    \end{equation}
\end{lemma}
We omit the proof of this lemma, as it follows from the definition of the binomial coefficients. 
It is simply a count of the number of ways we can write each entry $s_{ij+}$ in the marginal table as a sum of $K$ entries in the data table. 

\begin{remark}\rm
    We can find $s_{ij+}$ from the solutions of the Diophantine equation, since $s_{ij+}=x_j m_j c_{ij}$.

    With real data in mind, however, we might have to alter the formulas.  Specifically,
    the above formulas assume that the marginals $s_{ij+}$ are integers, but with real data due to possible rounding of observed conditional probabilities, the computed $s_{ij+}$'s may also be rounded.
    Recall that the Gamma function is defined so that $\Gamma(n)=(n-1)!$ for all integers $n$.  Since the binomial coefficient in (\ref{eq:binom1coef}) can be written
    in terms of factorials, if we replace $s_{ij+}$ with a real number instead of an integer, we get:
 \begin{equation}
        |\mathcal F_{AB}| = \prod_{1\leq i\leq I, 1\leq j\leq J}  \frac{\Gamma(s_{ij+} + K)}{K! \Gamma(s_{ij+} +1)} .
 \end{equation}
\end{remark}
For an example, see Section~\ref{sec:examples}.

We can use this formula to derive the exact size of the table space given observed conditionals.
\begin{corollary}[Exact count of data tables given conditionals]
\label{cor:exactTotalTableCount}
      The number of possible $r$-way tables given observed conditionals $[A|B]$ is
      \begin{equation}
        |\mathcal F_{A|B}| = \sum_{i=1}^m |\mathcal F_{AB}(\mathfrak p_i)| ,
    \end{equation}
    where  
    $m$ is the number of integer solutions to (\ref{eq:diophantine}), and 
    each $|\mathcal F_{AB}(\mathfrak p_i)|$ can be computed using Lemma \ref{lm:tableCountFormula}.
\end{corollary}
\begin{proof}
    The claim follows by Lemma \ref{lm:tableCountFormula} and Corollary \ref{cor:tableSpaceDecompositionByMarginals}.
\end{proof}
 A $tablecount()$ function in R implements the above results and gives the corresponding counts. In practice, however, it may be computationally difficult to obtain the number of solutions to the Diophantine equation exactly.
One remedy is provided by approximating the number of those solutions. Then, this approximation can be extended to
give an approximate size for the table space $\mathcal F_{A|B}$.  By {\it approximation}  we mean a Riemann sum approximation of the integral which calculates the volume of a polytope for fixed $N$. 
We deal with the number of marginal tables first, returning to the notation of Lemma \ref{lm:structureOfSolnsOfDioph}:

\begin{proposition}[Approximate count of marginal tables given conditionals]
    \label{prop:approxNumOfSolnsDiophantine-viaAlgebra}
    Given observed conditionals $[A|B]$, the number of possible marginal tables $[AB]$ is
    approximately
    \begin{equation}\label{eq:geovol}
        |\mathcal F_{A|B}|_{AB}  \approx  \ds\frac{N^{J-1}gcd(m_1,m_2,...,m_J)}{(J-1)!\prod\limits_{i=1}^{J}{m_i}}.
    \end{equation}
This approximation may also be given by a Dirichlet integral 
   \begin{equation}\label{eq:intvol}
        |\mathcal F_{A|B}|_{AB} \approx \frac{gcd(m_1,...,m_J)}{m_J}\int\limits_{(x_1,...,x_{J- 1} ) \in \mathcal{M}}   1  dx_1 dx_2  \cdot  \cdot  \cdot dx_{J - 1},
    \end{equation}
    where $\mathcal{M}$ is the projection of the marginal polygon onto the $x_1x_2...x_{J-1}$-plane.
\end{proposition}

A simple algebraic proof of this result can be found in Appendix \ref{sec:proof-of-prop:approxNumOfSolnsDiophantine-viaAlgebra}.
Note that by Theorem \ref{thm:diophantineSolutions=Marginals}, the number of possible marginal tables equals the 
number of positive integer solutions of Equation (\ref{eq:diophantine}).
Formula in equation (\ref{eq:geovol}) uses a geometric approach via volumes of cells in the lattice; the second formula in (\ref{eq:intvol}) realizes the same approximation using the integral formula for volumes. Section~\ref{sec:examples} illustrates the use of these approximation formulas. 

\begin{corollary}[Approximate count of data tables given conditionals]
    \label{cor:approxTableCount-Integral}
The number of possible $r$-way 
tables in $\mathcal F_{A|B}$ is approximately
\begin{equation}\label{eq:countapprox}
\frac{gcd(m_1,...,m_J)}{m_J}\int\limits_{(x_1,...,x_{J- 1} ) \in \mathcal{M}} {\prod\nolimits_{i,j} {\frac{{\Gamma (x_j m_j c_{ij}  + |
C|)}}{{\Gamma (|C|) \cdot \Gamma (x_j m_j c_{ij}  + 1)}}} } dx_1 dx_2  \cdot  \cdot  \cdot dx_{J - 1} ,
\end{equation}
where $\mathcal{M}$ is the projection of the marginal polygon onto the $x_1x_2...x_{J-1}$-plane.
\end{corollary}
\begin{proof}
	The claim follows from 	Lemma~\ref{lm:tableCountFormula} and Proposition~\ref{prop:approxNumOfSolnsDiophantine-viaAlgebra}. Note that the total number of $r$-way tables equals the sum over all possible marginals of the number of tables for a fixed marginal. The approximation comes from using the approximate count in equation (\ref{eq:countapprox}). 
\end{proof}

\section{Implications for cell bounds and Markov bases}
\label{sec:CellMarkov}

\subsection{  Cell bounds}
\label{subsec:bounds}

There has been much discussion on calculation of bounds on cell entries given the marginals (e.g., see \cite{dobra:fien:2009} and related references), and to a limited extent the bounds given the observed conditional probabilities; e.g.,  see \cite{slavkovic2_2004} and \cite{smuck:slav:2008}. Such values are useful for determining the support of underlying probability distributions. In the context of data privacy, the bounds are useful for assessing disclosure risk; tight bounds imply higher disclosure risk. We can use the structure of the space of possible tables to obtain sharp integer bounds for the cell counts.     
Recall that we assume that observed conditional probabilities are exact. 

There are a number of different ways to get cell bounds: (1) using linear and integer programming to solve the system of linear equations of (\ref{eq:spaceoftables});  (2) using the result of equivalence of marginal and conditional fibers (c.f., Corollary~\ref{cor:equalspace}), the bounds are given by $0\leq s_{ijk} \leq s_{ij+}$;  and (3) using our decomposition result (c.f., Corollary~\ref{cor:tableSpaceDecompositionByMarginals}) to enumerate all possible marginal tables, and based on those get the cell bounds $min_l (s_{ij+})_l \leq s_{ijk} \leq max_l (s_{ij+})_l$, where $l$ is the number of possible marginal tables $AB$ given $A|B$.

Besides the above three methods for computing the exact cell bounds, there is a fourth method that computes approximate cell bounds by allowing arbitrary rounding of $P(A|B)=c_{ij}$.
The proof is straightforward: simply recall that 
$\sum_j x_jm_j=N$ and  $s_{ij+}=m_jx_jc_{ij}$. 


\begin{theorem}\label{th:boundsonmargin}
Given  $\mathcal{T}=\{P(A|B),N\},$ 
an approximate (relaxation) integer cell bounds are given by
\begin{equation}
m_j\cdot c_{ij} \le s_{ij + }  \le (N - \sum\limits_{t\neq j}{m_t} )
\cdot c_{ij}.
\end{equation}
\end{theorem}

Furthermore, an approximate number of values that $x_i$ can take is given by\\
\begin{equation}
\frac{{(N-\sum\limits_{j\neq i}{m_j}) \cdot(m_1 ,m_2 ,...,m_J)}}{{m_i\cdot(m_{1,} ...,m_{i - 1},m_{i + 1} ,...,m_J )}}. 
\end{equation}

These bounds  can be made sharper if we know the rounding scheme of $c_{ij}$'s. The effect of rounding on bounds and on calculating Markov bases given observed conditionals is of special interest, but we defer that work to a future study. Some preliminary results and discussion are provided in \cite{smuck:slav:zhu:2009} and \cite{lee-thesis}.

\subsection{Markov bases}
\label{subsec:MB}
In this section, we describe a structure for the Markov bases for the table space $\mathcal{F_T}$ as defined in \eqref{eq:spaceoftables}, resulting from the Corollary~\ref{cor:tableSpaceDecompositionByMarginals}, which could lead to their simplified computation. 

A set of minimal Markov moves allows us to build a connected Markov chain and perform a random walk over all the points in 
any given fiber. Thus, we can either enumerate or sample from the space of tables via Sequential Importance Sampling (SIS) or Markov Chain Monte Carlo (MCMC) sampling; e.g., see \cite{dobra2006dam} and \cite{chen2006sis}.
A Markov basis for a model, or for its design matrix,
is a set of moves that are guaranteed to connect all points with the same sufficient statistic. In a seminal paper by \cite{diaconis1998}, these bases were used for performing exact conditional inference over contingency tables given marginals.

\begin{definition} [Definition (\cite{diaconis1998}).]
    Let $T$ be a $d\times n$ matrix whose entries are nonnegative integers. Assume $T$ has no zero columns.
    In addition, denote by $\mathcal F_t$ the fiber for $t$, that is, the set of all $d$-tuple preimages of $t$ under the map defined by $T$:
    \[
         \mathcal{F}_t=\{f\in\mathbb{N}^d : Tf=t\},
    \]
    where $t$ is in $\mathbb{N}^d\backslash \{0\}$.

    A \emph{Markov basis} of $T$ is a set of vectors $f_1,\dots,f_L\in \mathbb Z^n$ with the following properties:
    First, the vectors must be in the kernel of $T$:
    \[
        Tf_i=0,\quad 1\leq i\leq L.
    \]
    Secondly, they must connect all vectors in a given fiber:
     for any $t\in\mathbb{N}^d\backslash \{0\}$ and any $f, g\in \mathcal{F}_t$,
      there exist $(\epsilon_1,f_{i_1}),...,(\epsilon_K,f_{i_K})$ with $\epsilon_i=\pm1$,
      such that
      \[
        g=f+\sum\limits_{j=1}^{K}{\epsilon_jf_{i_j}}
    \]
    and, at any step, we remain in the fiber:
    \[
        f+\sum\limits_{j=1}^{a}{\epsilon_jf_{i_j}}\geq0\mbox{ for all $a$ such that } 1\leq a\leq K.
    \]
\end{definition}
Note that the definition of a Markov basis does not depend on the choice of $t$; it must connect \emph{each} of the fibers.

In our problem, $T$ is the matrix $M$ in equation \eqref{eq:spaceoftables}.
Thus, the fiber $\mathcal F_t$ contains the space of possible data tables that satisfy the constraints described in \eqref{eq:spaceoftables} for the given vector $t$.
Theorem 3.1. in \cite{diaconis1998} is considered one of the fundamental theorems in algebraic statistics and stats that a Markov basis of $T$ can be calculated as a generating set of the toric ideal $I_T$ for the design matrix $T$ of the model; for an introduction to toric varieties of statistical models see \cite{drton2009lectures}. 

There are a number of algebraic software packages for computing generating sets of toric ideals, and thus the Markov bases, but the most efficient to date is 4ti2 (\cite{4ti2}). 
Sometimes, though, the matrix $M$ can be large, and the computation may take too long.
To alleviate some of the computational problems with contingency tables in practice, we use our table-space decomposition result (c.f. Corollary ~\ref{cor:tableSpaceDecompositionByMarginals}) to split the Markov basis into two sets. This could allow for parallel computation of the Markov sub-bases. 
\begin{corollary}\label{cor:MBTwoSubsetsofMoves}
The Markov basis for the space of tables given the conditional can be split into two sets of moves:
\begin{itemize}
\item[1)] the set of moves that fix the margin, and
\item[2)] the set of moves that change the margin.
\end{itemize}
\end{corollary}
\begin{proof}
By Corollary  \ref{cor:tableSpaceDecompositionByMarginals}, the fiber $\mathcal F_{A|B}$ of tables given the conditional is a disjoint union of the sub-fibers $\mathcal F_{AB}(\mathfrak p_i)$ given the fixed marginals represented by the points $\mathfrak p_i$, for $i=1,\dots, m$.
By definition, the set of Markov moves consisting of the moves that change the margin connect the sub-fibers $\mathcal F_{AB}(\mathfrak p_i)$, for $i=1,\dots, m$.
Thus, the Markov basis connecting all of $\mathcal F_{A|B}$ consists of the moves connecting each sub-fiber $\mathcal F_{AB}(\mathfrak p_i)$ (the first set of moves) and the moves connecting each sub-fiber to another (the second set of moves).
\end{proof}

The moves that fix the margins have been studied in the algebraic statistics literature; for some recent advances in that area, see \cite{AokiTakemura03}, \cite{aoki2008minimal}, \cite{DeLoera06}, and references given therein. Most recently,  \cite{Dobra2012} provided an efficient algorithm to dynamically generate the moves given the margins. Less work has been done on studying Markov bases given observed (estimated) conditionals, e.g., see \cite{sesa-thesis, lee-thesis}. Since we know, by Theorem \ref{thm:diophantineSolutions=Marginals}, that the margins correspond to solutions to the Diophantine equation (\ref{eq:diophantine}),
we can find the latter set of moves by computing the Markov basis for the coefficient matrix of the Diophantine equation.

The number of Markov basis elements for this matrix seems to be small. More specifically, computations suggest the number of Markov basis elements that change the margin is as small as possible: 
\begin{conjecture}\label{prop:conjecture}
    In the case of small conditionals (i.e., $C \neq \emptyset$), the coefficient matrix of the Diophantine Equation (\ref{eq:diophantine})
    has a Markov basis consisting of $J-1$ elements, where $J-1$ is the dimension of the underlying lattice.
    In other words, the corresponding toric ideal equals the lattice basis ideal.
\end{conjecture}
Note that the assumption $C\neq \emptyset$ is necessary, as the Example in \ref{sec:example:conj-not-true-full-conditionals} shows. \\
If the conjecture were true, it would imply the following on the size of the entire Markov basis:
\begin{conjecture}\label{cor:MBMinimalNumberofElements}
    A minimal Markov basis of the matrix $M$ in (\ref{eq:spaceoftables}) contains $|B|-1+(|C|-1)\times|B|\times|A|$ elements.
\end{conjecture}

Supporting examples for the above conjectures are included in Section~\ref{sec:examples}. \\

On a related note, Peter Malkin has shown (in personal communication) that under certain assumptions, the number of solutions to the \emph{homogeneous} linear Diophantine equation is exactly the dimension of the lattice, where by homogeneous we mean the right-hand side is zero:
 Let $D$ be the minimal size of all $det(L_i)$, where $L_i$ is the projection of the lattice $L$ onto all variables except the $i^{th}$ variable.
 In general, a $(k-1)$-dimensional lattice in $k$ variables has a Markov basis of size at least $(k-1)$ and at most $(k-2)D+1$.
 Note that if $D=1$, then the upper bound is $k-1$.
 The size of the Markov basis for the $k-1$-dimensional lattice can be obtained as a consequence of a result in \cite{StWeZi} and the Project-and-Lift method from \cite{HeMa}. 
  Namely, Proposition 4.1. of \cite{StWeZi} states that the maximal size of a Gr\"obner and thus a Markov basis for a $k$-dimensional lattice $L$ in $k$ variables is at most $(k-1)det(L)+1$. They state without proof that $(k-2)det(L)+k+1$ is also an upper bound.
  The Project-and-Lift method is the one implemented in 4ti2 (\cite{4ti2}).

 Even though we cannot show that $D=1$ holds, the conjecture above says that the size of the Markov basis is actually as small as possible.
 It would be of interest to obtain bounds tighter than the general one in the case of a Diophantine equation arising from the study of the table space.
 For more about the sizes of Markov bases and computing them, see \cite{Malkin-thesis}.

\subsection{Extension of relations to marginals via DAGs}
\label{subsec:dag}

Given the marginals only, \cite{dobra2003markov} and  \cite{dobra2000bounds} have used graphical models in computing Markov bases and for calculating bounds for disclosure risk assessment. In this section, we provide extensions to those results, to the bounds in Section~\ref{subsec:bounds} and to Problem \ref{mainProb} by considering combinations of multiple conditional arrays and their relations to corresponding marginals via Directed Acyclic Graphs (DAGs); see Section~\ref{subsec:exdag} for an example. 

A DAG
$\mathcal{G}=\{\mathcal{V},\mathcal{E}\}$ consists of a set of
nodes $V=\{v_{1},...,v_{r}\}$ and a set of directed edges,
$(v_{i}, v_{j})\in E$, that link the ordered pairs of distinct
nodes $v_i$ ({\it the parent}), and $v_j$ ({\it the child}) in $V$, and there are no {\it cycles}.  
A DAG satisfies the {\it Wermuth condition}
(\citet{Whittaker1990}) or is {\it perfect}
(\citet{Lauritzen1996}) if no subgraph has {\it colliders}, that
is, if no child has parents that are not directly connected. A graph
$\mathcal{G}^{u}=\{\mathcal{V},\mathcal{E}^{u}\}$ is called {\it
undirected} if the edges are undirected (lines), that is, if
$(v_i, v_j) \in E$ then $(v_j,v_i)\in E$. A {\it moral graph}
$\mathcal{G}^{m}=\{\mathcal{V},\mathcal{E}^{m}\}$ is the
undirected graph on the same vertex set as $\mathcal{G}$ and with
the same edge set $\mathcal{E}$ including all edges that would be
necessary to eliminate forbidden Wermuth configurations in $\mathcal{G}$.

If the random
variables $X_1,...,X_r$ are nodes of the graph
$\mathcal{G}$, then the graph
represents dependencies among these variables. More
specifically, $\mathcal{G}$ defines
the set of probability distributions over the sample space that
obeys the {\it directed Markov properties} and factorizes the joint distribution, 
\begin{equation}
f(x_{1},x_{2},...,x_{r}) =  \prod_{x\in\mathcal{V}}f(x|pa(x))=f(x_{1})f(x_{2}|x_{1})...f(x_{r}|x_{r-1},x_{n-2},...,x_{1}).
\end{equation}
There are many cases when the joint distribution over the contingency
table has a graphical representation. In some of these cases, a
set of conditionals and marginals will factor the joint according
to a DAG representation. Given such a set that also satisfies the Wermuth
condition, there is an equivalent undirected graph representation
of the same set. In that case, the generalized Problem~\ref{mainProb}  is reduced to one of knowing a set of marginals, and the bounds are those given by \cite{dobra2000bounds, dobra:fien:2009}. The following results hold for any $r$-way table.

\begin{theorem}\label{th:DAGmargin}
Let $ \mathcal{T}$ be a set of conditional and marginal distributions
inducing bounds on the cell entries. Let $\mathcal{G}$ be a DAG,
and $\mathcal{G}^{u}$ the undirected graph associated with $ \mathcal{T}$.
When $\mathcal{G}$ satisfies the Wermuth condition, the bounds
imposed by $ \mathcal{T}$ reduce to the bounds imposed by a set of
marginals associated with $\mathcal{G}^{u}$.
\end{theorem}
\begin{proof}
This result follows from well-known properties of a DAG and more
specifically from the Markov theorem for directed independence
graphs (\citet{Whittaker1990, Lauritzen1996}).  The theorem states
that the DAG possesses the Markov properties of its associated moral
graph. Therefore, there is an equivalence of the set of edges for
$\mathcal{G}^m$ and $\mathcal{G}^u$.  The directed edges in the
DAG carry independence statement  information on a sequence of
marginal distributions, while the undirected graph describes the
independence statements on a single conditional.  Since the edge
sets are equivalent, the DAG then gives the equivalent information
on the joint as its associated undirected graph.
\end{proof}

\begin{corollary}
Let $\mathcal{G}^{m}$ be the moral graph associated with
$\mathcal{G}$. If $\mathcal{G}^{m}=\mathcal{G}^{u}$, then the
bounds induced by a set $ \mathcal{T}$ are equivalent to the bounds induced
by the set of marginals associated with $\mathcal{G}^{u}$.
\end{corollary}

An interesting link between bounds on cells in the contingency
tables, DAGs, and Markov bases is indicated by the next result.

\begin{corollary}\label{th:DAGMbasis}
Let $ \mathcal{T}$ be a set of conditional and marginal distributions. Let
$\mathcal{G}$ be a DAG  and $\mathcal{G}^{u}$ the undirected graph
associated with $ \mathcal{T}$. When $\mathcal{G}$ satisfies the Wermuth
condition, the Markov basis describing $ \mathcal{T}$ under the same ordering
is the same Markov basis induced by a set of marginals associated
with $\mathcal{G}^{u}$.
\end{corollary}

\begin{proof}
The claim follows from Corollary~\ref{cor:MBTwoSubsetsofMoves}. 
\end{proof}

It is possible that similar results, with discrete random variables, could be derived for the chain graphs and ancestral graphical models (e.g., \cite{richardson2002}) which are generalization of the directed and undirected graphs. This is an interesting topic for future research. 

\section{Examples}\label{sec:examples}
\vspace{-0.1in}
In this section we illustrate the results described in the preceding sections through  analysis of a series of simple contingency tables.
We show how to use our initial R~\citep{R} implementation of the formulas from Sections~\ref{sec:mathOfTableSpace} and~\ref{sec:CellMarkov}. We also perform our analyses using the well-establish and free algebraic software LattE macchiato \citep{latte} which relies on an implementation of the Barvinok's algorithm \citep{barvinok2010approximation} for counting and detecting lattice points inside convex polytopes. 
In statistical literature, LattE has been mostly used for counting the number of tables given the margins.

\vspace{-0.1in}
\subsection{A $2\times 2\times 2$ Example}\label{subsec:mp3}
\vspace{-0.1in}
Consider a fictitious  $2\times 2 \times 2$ table that cross-classifies a randomly chosen sample of 50 college students by their {\em Gender}, illegal {\em Downloading} of MP3 files, and the dorm {\em Building} they live in; see counts in Table~\ref{tab:mp3_bgd}. We  use shorthand $G$ for $Gender$, $D$ for $Downloading$, and $B$ for $Building$ variable.
\begin{table}[htdp]
\small
\caption{A $2\times2\times2$ table of counts of  illegal MP3 downloading by gender and a residing building. The value in the brackets are linear relaxation bounds and sharp integer bounds given released conditional $[D|G]$ and marginal $[DG]$, respectively.}
\begin{center}
\begin{tabular}{cc|cc|c}
&&Download&&\\
Building&Gender&Yes&No&Total\\\hline
I&Male&8 [0,29.4] \color{blue}[0,27] \color{red}[0,15] & 4 [0,19.6] \color{blue}[0,18] \color{red}[0.10]&12\\
I&Female&2 [0,9.8] \color{blue}[0,9] \color{red}[0,5] & 9 [0,39.2] \color{blue}[0,36] \color{red} [0,20]&11\\
II&Male&7 [0,29.4] \color{blue}[0,27] \color{red}[0,15] & 6 [0,19.6] \color{blue}[0,18] \color{red}[0.10]&13\\
II&Female&3 [0,9.8] \color{blue}[0,9] \color{red}[0,5] &11 [0,39.2] \color{blue}[0,36] \color{red} [0,20]&14\\\hline
&Total&20&30&50\\
\end{tabular}
\end{center}
\label{tab:mp3_bgd}
\end{table}

\begin{table}[htdp]
\small
\parbox{.45\linewidth}{
\centering
\caption{$[GD]$ Marginal table of illegal MP3 downloading, and integer bounds given released $[D|G]$ and $N=50$.}
\begin{center}
\begin{tabular}{c|cc|c}
&Download&&\\
Gender&Yes&No&Total\\\hline
Male&15 [3,27] &10 [2,18] &25\\
Female&5 [1,9] &20 [4,36] &25\\\hline
Total&20&30&50\\
\end{tabular}
\end{center}
\label{tab:mp3_gd}
}
\hfill
\parbox{.45\linewidth}{
\centering
\caption{$[D|G]$ Table of conditional probabilities with reduced fractions and [rounded probability].}
\begin{center}
\begin{tabular}{c|cc}
&Download&\\
Gender&Yes&No\\\hline
Male&$\frac{15}{25}=\frac{3}{5}$ $[0.6]$&$\frac{10}{25}=\frac{2}{5}$ $[0.4]$\\
Female&$\frac{5}{25}=\frac{1}{5}$ $[0.2]$&$\frac{20}{25}=\frac{4}{5}$ $[0.8]$\\\hline
Total&20&30\\
\end{tabular}
\end{center}
\label{tab:mp3_d|g}
}
\end{table}

The survey administrator has the full information on the $[BGD]$ table, but due to confidential nature of the data, would like to consider releasing only partial information to public such as the marginal counts 
$[DG]$ as in Table~\ref{tab:mp3_gd} or the grand total $50$ and the small conditional $P(Download|Gender)$ as in Table~\ref{tab:mp3_d|g}. This requires comparison of the space of tables $\mathcal F_{DG}$, which based on Lemma~\ref{lm:tableCountFormula} has $16\times 11\times 6 \times 21=22176$ possible $[BGD]$ tables, with the space of tables $\mathcal F_{D|G}$. 

The reference set $\mathcal F_{D|G}$ consists of tables that are solutions to the following:

\begin{footnotesize}
\[
\left\{ \begin{array}{l}
 \left[ {\begin{array}{*{20}c}
   1 & 1 & 1 & 1 & 1 & 1 & 1 & 1  \\
   2 & { - 3} & {} &{} & {2} & {-3} & {} & {}   \\
   {} & {} & {4} & {-1} & {} & {} & 4 & { - 1}   \\
  \end{array}} \right]\mathbf{n} = \left[ {\begin{array}{*{20}c}
   {50}  \\
   0  \\
   0  \\
\end{array}} \right] \\
 n_1  + n_2  + n_5  + n_6  > 0 \\
 n_3  + n_4  + n_7  + n_8  > 0 \\
 All\quad n_i's \quad are \quad nonnegative \quad integers
 \end{array} \right\}
\]
\end{footnotesize}

This is part of a 5-dimensional lattice inside the $\mathcal{R}^2$. Then equation (\ref{eq:diophantine}) of Theorem
\ref{thm:diophantineSolutions=Marginals} for this example is $5x_1+5x_2=50,$ and it has 9 positive integer solutions: $\{(x_1=i, x_2=10-i)| 1\leq i\leq9\}$. Thus, there are $9$ different $[DG]$ marginals, 
which, by Theorem \ref{cor:equalspace}, means that the space of
tables given the small conditional $[D|G]$ and the
grand total is different
from the space of tables given the corresponding marginal counts. In fact, the space is larger:
$|\mathcal F_{D|G}| > |\mathcal F_{DG}|$.  More specifically, Corollary \ref{cor:exactTotalTableCount}
for $m=9$ provides the table count:
 $ |\mathcal F_{D|G}| = \sum_{m=1}^9 |\mathcal F_{DG_m}|=128676$. In R, we invoke function {\em tablecount}$(M,2)$ where $M$ is any one of 9 possible marginal tables $[DG]$. Notice that this formulation does not allow any row of $[G]$ to have a total of zero counts. If such tables were to be allowed, then the total number of possible 3-way tables would be $128676+651+451=129778$ where $651$ and $451$ are the numbers of possible 3-way tables given the $[DG]$ when one of the rows of $[G]$ is equal to zero.

To approximate the number of marginal tables $[DG]$, one can use the formula from equation (\ref{eq:intvol}) in Proposition~\ref{prop:approxNumOfSolnsDiophantine-viaAlgebra} to count the number of corresponding
solutions to the Diophantine equation as $\frac{50\text{gcd}(5,5)}{5\times5}=10$. Then, we can use the integral formula from Corollary \ref{cor:approxTableCount-Integral}, which could be evaluated, say, using Maple, to estimate the size of the total table space given the conditionals as 
$\frac{\text{gcd}(5,5)}{5}\int\limits_0^{10} {(3x+1)(2x+1)(10-x+1)(40-4x+1)}=129676.7$.


Since more than one possible margin is consistent with the given conditional and grand total, clearly $\mathcal F_{DG}$ is strictly contained in $\mathcal F_{D|G}$. This can also be seen by computing the cell bounds on the cell entries of $[BDG]$ contingency table. In Table~\ref{tab:mp3_bgd}, given $\mathcal F_{D|G}$,  the linear relaxation cell bounds and the exact integer bounds are given in the black and blue brackets, respectively.  Given $\mathcal F_{DG}$,  the exact cell bounds are in red brackets. The idea is that the wider bounds offer more protection. These bounds are obtained by direct optimization for each given constraint.  However, the results of Section~\ref{sec:CellMarkov} show a computational shortcut to obtaining bounds given $[D|G]$ and $N=50$ by using already established results on bounds of cell entries given the marginals. First, by Theorem~\ref{th:boundsonmargin} we obtain bounds on the missing margin $[DG]$ (see Table~\ref{tab:mp3_gd}). Next, we combine this with a well-known fact that given one marginal $s_{ij+}$, the bounds  on each cell entry of the 3-way table are $0\leq n_{ijk}\leq s_{ij+}$. Thus, the bounds for $n_{ijk}$ are between 0 and the upper bound found for the missing marginal table. For example, for the cell $(1,1,1)$, the $3\leq s_{11+} \leq 27$, and $0\leq n_{111} \leq 27$; these are the bounds given in the blue brackets in Table~\ref{tab:mp3_bgd}.

It has been observed in the literature already that the above-described bounds have gaps. That is, not all values within the interval are possible.
This observation 
is particularly important for assessing disclosure risk with contingency tables. By enumerating all possible marginal tables, we learn both the number of all possible $r$-way tables, and the values in the cell counts of those tables.  We can obtain such tables quickly by using the {\em solvequick()} function. For example, $solvequick(c(5,5),50)$ gives a vector of all possible $G$ margins that we conditioned on in $[D|G]$. To get $[DG]$ margins, compute $m_j\times b\times c_{ij}.$ 

Next, we calculate a Markov basis for fixed $[D|G]$ using $4ti2$. If Conjecture~\ref{prop:conjecture} is true, then so is Corollary~\ref{cor:MBMinimalNumberofElements}, and there should be $5=|G|-1+(|B|-1)\times |G| \times |D|=1+1\times2\times2$ Markov moves. 
Our computation finds exactly $5$ moves: \begin{footnotesize}
\[
 \left( {\begin{array}{*{20}c}
     3 & { 2} & {-1} &{-4}&{0} & {0} & {0} & {0}   \\
     {1} & {0} & {0} & {0} & {-1} & {0} & {0} & {0}   \\
   {0} & {1} & {0} & {0} & {0} & {-1} & {0} & {0}   \\
 {0} & {0} & {1} & {0} & {0} & {0} & {-1} & {0}   \\
 {0} & {0} & {0} & {1} & {0} & {0} & {0} & {-1}   \\

  \end{array}} \right).
\]
\end{footnotesize}

In accordance with Corollary~\ref{cor:MBTwoSubsetsofMoves}, the last 4 moves correspond to a set of moves that fix the $[DG]$ margin, while the first move changes the margin $[DG]$, but keeps the $N$ fixed. From the first element, $n_1^{3}n_2^{2}-n_5^{1}n_6^{4}$, by summing the exponents in each monomial, we can deduce exactly the amount by which a count in each level of the margin we condition on changes. In this example, each marginal count of $[G]$ changes by a count of $5$. Thus, with the sample size $N=50$, the upper bound for the solution to equation~(\ref{eq:diophantine}) for the number of possible marginals $[G]$, and thus of $[DG]$, is $10$. 

A related example, providing more details and implications of when a Diophantine equation has only one solution is available in the supplementary documents at \url{http://www.stat.psu.edu/~sesa/cctable}. 
\subsection{A $3\times 2\times 2$ table with zero counts}

In this section, we apply our derived results to a $3\times2\times2$ table (see Table~\ref{ex:322table}) with zero counts, and show the convergence of exact and approximate results.

\small
\begin{table}[htdp]
\caption{A $3\times2\times2$ Table}
\begin{center}
\begin{tabular}{cc|cc|c}
 &  & C=1 & C=2 & Total\\
\hline
A=1 & B=1& 10 & 20 & 30  \\
A=1 & B=2 & 10 & 20 & 30\\
A=2 & B=1 & 20 & 0 & 20\\
A=2 & B=2 & 0 & 40 & 40 \\
A=3 & B=1 & 0 & 30 & 30 \\
A=3 & B=2 & 30 & 60 & 90\\
\hline
& Total & 70 & 170 & 240\\
\end{tabular}
\end{center}
\label{ex:322table}
\end{table}

\normalsize
\subsubsection{Small conditional $B|A$ and $N$}
Consider that we do not observe the original table, and the only available information is $\mathcal{T}=\{Pr(B|A), N=240\}$; the sample values are given in Table~\ref{ex:32table}. 
\vspace{-0.3in}
\begin{table}[htdp]
\caption{Left panel: Observed counts of the $[AB]$ marginal table, and notation for when those counts are missing. Right panel: Observed conditional probabilities $[B|A]$ based on values in Table~\ref{ex:322table}.}
\begin{center}
\begin{tabular}{l|l|l}
& B=1 & B=2  \\
\hline A=1 & 30 [x] & 30 [x] \\
A=2 & 20 [y]  & 40 [2y] \\
A=3 & 30 [z]  & 90 [3z] \\
\hline
\end{tabular}
\hspace{0.5in}
\begin{tabular}{l|l|l}
& B=1 & B=2  \\
\hline A=1 & 1/2 & 1/2 \\
A=2 & 1/3  & 2/3 \\
A=3 & 1/4  & 3/4 \\
\hline
\end{tabular}
\end{center}
\label{ex:32table}
\end{table}

By Theorem~\ref{thm:diophantineSolutions=Marginals}, the linear Diophantine equation that characterizes all possible missing [$AB$] margins is
\begin{equation} \label{eq:ex322dioph}
2x+3y+4z=240.
\end{equation}
Using our R code, e.g., $solvecount(c(2,3,4),240),$ we learn that there are $1141$ possible $A$ marginals consistent with the provided information. Since the triplets $(x,y,z)$ are in $1$-to-$1$ correspondence to $[AB]$ margins (see Table~\ref{ex:32table}), there are $1141$ missing $[AB]$ marginals consistent with the provided information. Furthermore, {\em solvequick}$(c(2,3,4),240)$ lists all positive  integer solutions to Equation (\ref{eq:ex322dioph}), and from there we easily obtain all corresponding $[AB]$ margins.

We are ultimately interested in finding all possible 3-way tables consistent with given information, i.e, solutions to the following system

\begin{footnotesize}
\[
\left\{ \begin{array}{l}
 \left[ {\begin{array}{*{20}c}
   1 & 1 & 1 & 1 & 1 & 1 & 1 & 1 & 1 & 1 & 1 & 1  \\
   1 & { - 1} & {1} & {-1} & {} & {} & {} & {} & {} & {} & {} & {}  \\
   {} & {} & {} & {} & {2} & {-1} & {2} & {-1} & {} & { } & {} & {}  \\
   {} & {} & {} & {} & {} & {} & {} & {} & {3} & {-1} & 3 & { - 1}  \\
\end{array}} \right]X = \left[ {\begin{array}{*{20}c}
   {240}  \\
   0  \\
   0  \\
   0  \\
\end{array}} \right] \\
 n_1  + n_2  + n_3  + n_4  > 0 \\
 n_5 + n_6  + n_7  + n_8  > 0 \\
 n_9  + n_{10}  + n_{11}  + n_{12}  > 0 \\
 All\quad n_i's \quad are \quad nonnegative \quad integers
 \end{array} \right\},
\]
\end{footnotesize}

which is part of a 8-dimensional lattice inside the $\mathcal{R}^{12}$.  The exact number of possible 3-way tables can be obtained by Corollary~\ref{cor:exactTotalTableCount},  $ |\mathcal F_{B|A}| = \sum_{m=1}^{1141} |\mathcal F_{AB_m}|.$ In R, we invoke $format(tablecount(M,2),digits=22)$, which gives $1187848498271$ possible $[ABC]$ contingency tables. 

Next, we demonstrate in a little more detail and following the proof of Proposition~\ref{prop:approxNumOfSolnsDiophantine-viaAlgebra}, how to set up the integrals to calculate the approximate number of solutions. 
Recall that a marginal table [$AB$] corresponds to a triple
$(x,y,z)$. Note that $z=(240-2x-3y)/4$.
Thus, for each
marginal table, 
the number of possible tables that have this margin is $$(x+1)^2(y+1)(2y+1)(\frac{240-2x-3y}{4}+1)(3\frac{240-2x-3y}{4}+1).$$ After summing over all possible $(x,y)$, we get the count of
all possible $[ABC]$ tables:
$$\sum_{(x,y)\in\mathcal{M}}{(x+1)^2(y+1)(2y+1)(\frac{240-2x-3y}{4}+1)(3\frac{240-2x-3y}{4}+1)}$$
where $\mathcal{M}$ is the projection of all possible triple
$(x,y,z)$ onto the $xy$-plane. As discussed in the proof of Proposition~\ref{prop:approxNumOfSolnsDiophantine-viaAlgebra}, notice that $\mathcal{M}$ is a part of a lattice
whose unit cell has an area of 4/gcd(2,3,4). 
Thus, the number of possible solutions is approximately $1.188479935\times10^{12}$ by solving the following
\begin{small}
\[
\frac{1}{4} \int\limits_{\rm{0}}^{80} {\int\limits_0^{\frac{{240 -
3y}}{2}} {(x + 1)^2 (y + 1)(2y + 1)(\frac{{240 - 2x - 3y}}{4} + 1)(3
\cdot \frac{{240 - 2x - 3y}}{4} + 1)} } dxdy. 
\]
\end{small}

 The ratio of the exact solution to the approximate solution, for either counting the missing margin or the $r$-way table, is $1+O(1/N)$. For this example, we compute exact and approximate number of tables while varying  the grand total $N$. Table \ref{tab:marginratio} summarized the results for the missing marginal $[AB]$, and Table~\ref{tab:allratio} lists the exact number and approximate number of $[ABC]$ tables for different values of the total sample size. Numerical experiments show evidence that our approximation is sharper for equations with fewer unknowns, and/or when $N$ is
much larger than the coefficients in the equation.  For the small number of margins, the approximation does not work well. 

\begin{table}[htdp]
\caption{Exact and approximate number of missing marginal tables $[AB]$.}
\label{tab:marginratio}
\begin{tabular}{l|l|l}
\hline & Exact Count & Approximation  \\
\hline
N=24 &  $7$ & $12$ \\
N=240 & $1141$ & $1200$ \\
N=2400 & $119401$  & $ 120000$ \\
N=24000 & $11994001$  & $12000000
$ \\
\hline
\end{tabular}
\end{table}
\begin{table}[htdp]
\caption{Exact and approximate number of missing tables $[ABC]$.}
\label{tab:allratio}
\begin{tabular}{l|l|l}
\hline & Exact Count & Approximation  \\
\hline
N=24 &  $52937$ & $65150$ \\
N=240 & $1187848498271$ & $1.188479935\times10^{12}$ \\
N=2400 & $96999660430647444101$  & $ 9.699971869\times10^{19}$ \\
N=24000 & $9501190342113804461451781001$  &
$9.501190349\times10^{27}
$ \\
\hline
\end{tabular}
\end{table}

Next, we calculate a Markov basis for fixed $[B|A]$ using $4ti2$. According to Corollary~\ref{cor:MBMinimalNumberofElements}, there should be 8 elements in this basis. A Markov basis for this example is given below. In accordance to Corollary~\ref{cor:MBTwoSubsetsofMoves}, the last 6 moves correspond to a set of moves that fix the $[AB]$ margin, and the first two moves change the margin $[AB]$ while keeping $N$ fixed. As noted before, the sum of the exponents in the monomial tells us by how much the margin $[A]$ can change.
 
\begin{footnotesize}
\[
 \left( {\begin{array}{*{20}c}
-2 &-2& 0& 0& 0& 0& 0& 0& 1& 3& 0& 0\\
-3 &-3 &0 &0& 2 &4 &0 &0 &0 &0& 0& 0 \\
-1 &0 &1 &0 &0 &0 &0 &0 &0 &0 &0 &0 \\
0 &-1 &0 &1 &0 &0 &0 &0 &0 &0 &0 &0 \\
0 &0 &0 &0 &-1 &0 &1 &0 &0 &0 &0 &0 \\
0 &0 &0& 0 &0& -1 & 0& 1 &0 &0 &0 &0 \\
0 &0 &0 &0 &0 &0 &0 &0 &-1 &0& 1& 0 \\
0 &0& 0& 0& 0& 0& 0& 0& 0& -1& 0& 1 \\
\end{array}} \right)
\]
\end{footnotesize}

\subsubsection{Full conditional $A|BC$ and $N$}
\label{sec:example:conj-not-true-full-conditionals}
Before considering the release of other partial conditionals, we next demonstrate how some of our results also hold for the full conditional. First, if the only information available about the original table are the observed conditional rates, e.g., $[A|BC]$, and $N$,  as indicated in Section~\ref{subsec:MB}, we only need to solve a linear Diophantine equation to find the total number of possible 3-way tables, e.g., $$3x_1+4x_2+5x_3+6x_4=240.$$ We would typically count the number of possible solutions by setting up the full constraint matrix in LattE (e.g., see Appendix~\ref{sec:app-code}), but now we can simply apply $solvequick(c(3,4,5,6),240)$ in R. 
The number of possible tables is $5715$, which corresponds to the number of possible $[BC]$ margins. Second, notice that the $[A|BC]$ conditional rates have zero values, e.g., cell $(2,2,1)$ since the original cell has a zero count. However, the presence of zeros does not affect our computation since we are not conditioning on margins with zero counts. 

Last, the Markov basis has the following $4$ elements, all of which change the $[ABC]$ margin:   
 
\begin{footnotesize}
\[
 \left( {\begin{array}{*{20}c}
-2 &0 &0& 1 &-4 &0 &0 &3 &0 &0 &0 &2 \\
-3 &2 &1 &0 &-6 &0 &0 &0 &0 &3 &3 &0 \\
-2 &4 &-1& 0 &-4 &0 &0 &0 &0 &6 &-3 &0\\
-1 &-2& 2 &0 &-2 &0& 0 &0 &0 &-3& 6 &0 \\
\end{array}} \right).
\]
\end{footnotesize}

Conjecture~\ref{prop:conjecture} about the number of elements in the basis, however, does not hold here because we are using full conditionals, that is, $C=\emptyset$.
As supported by other examples, this conjecture seems true for \emph{small} conditionals only.

\subsubsection{Partial conditional $B|C$ and $N$}
Here we briefly consider a case where the missing marginal has more than two levels.  Let the available information be the sample size and the small conditional $[B|C]$ with the missing variable $[A]$ that has 3 levels. The following Diophantine equation captures the information preserved by the sample size and $[B|C]$: $$7x_1+17x_2=240.$$ In R, the $solvequick(c(7,17),240)$ function obtains  only two possible non-negative integer solutions, that is, only two possible marginal tables $[BC]$. Then, running {\em tablecount(M,3)}, where $M$ is one of the $[BC]$ margins, tells us that there are total of $6130182419416$ $[ABC]$ tables. In this example, it is easy to check via LattE that Corollary \ref{cor:exactTotalTableCount} holds. We compute the number of $ABC$ tables given each $BC$ margin, and see that their sum is equal to the number we obtained via the {\em tablecount()} function. According to this corollary,  $ |\mathcal F_{B|C}| = \sum_{m=1}^2 |\mathcal F_{BC_m}|=4179685045536+1950497373880=6130182419416.$ It should be noted here that the function {\em tablecount(M,3)} gives the total number of $[ABC]$ tables regardless of which compatible $[BC]$ margin we use. 
The conjectures for the size of Markov bases hold here as well. We observe that there are 9 elements in a basis: 8 fix the $[BC]$ margin, and 1 changes the $[BC]$ margin. 

 \subsubsection{Combinations of partial conditionals and $N$}\label{subsec:exdag}
Let's assume that we observe $\mathcal{T}=\{P(B|A), P(C|A), P(A), N\},$ and recall that we assume that there exists a joint distribution from which we observed these compatible pieces. Then this collection can be graphically represented by a  DAG $\mathcal{G}$  that satisfies the Wermuth condition. This DAG and its corresponding undirected graph $\mathcal{G}^{u}$ are given in the picture below. By Theorem~\ref{th:DAGmargin} the bounds on the cell counts are the same as in the case of given margins $[AB]$ and $[AC]$. 
Based on Corollary~\ref{th:DAGMbasis}, the Markov bases will be the same, and so will the fibers $\mathcal{F}_{\tau}$ and $\mathcal{F}_{AB,AC}$. Note that these results capture the following  special case: if the model according to DAG is true, that is $B$ and $C$ are conditionally independent given $A$, then by the Wermuth condition we can uniquely specify the joint distribution, $P(A,B,C)=P(AB)P(AC):$

\begin{center}
\setlength{\unitlength}{1cm}
\begin{picture}(6,1)
$\mathcal{G}:$\put(0.4,0){$B$}\put(2.7,0.1){\vector(-2,0){2}}\put(2.9,0){$A$}\put(3.3,0.1){\vector(2,0){2}}\put(5.3,0){$C$}
\end{picture}
\end{center}

 \begin{center}
\setlength{\unitlength}{1cm}
\begin{picture}(6,1)
$\mathcal{G}^{u}:$\put(0.4,0){$B$}\put(2.7,0.1){\line(-2,0){2}}\put(2.9,0){$A$}\put(3.3,0.1){\line(2,0){2}}\put(5.3,0){$C$}
\end{picture}
\end{center}

Now assume that marginal $[A]$ is missing or hidden, and we only have partial information in the form of observed conditional frequencies $[B|A]$ and $[C|A],$ and sample size $N$. If there is a unique solution for the margin $[A]$, then there are unique two-way margins $[AB]$ and $[AC]$. By Theorem~\ref{th:DAGmargin} and Corollary~\ref{th:DAGMbasis} then this is equivalent to having information on two margins, and we can proceed by calculating the cell bounds, counting tables, and by sampling given the marginals.

Consider our running example from Table~\ref{ex:322table} but with $N=24$. Let $\mathcal{T}=\{P(B|A), P(C|A), N=24\},$ where the observed conditional values are the same as with $N=240$; e.g., for $P(B|A)$, see Table~\ref{ex:32table}. By Theorem~\ref{thm:diophantineSolutions=Marginals}, the equation that characterizes the missing marginal $[A]$ and thus $[AB]$ for $[B|A]$ is
\begin{equation}\label{eq:exb|a}
2x+3y+4z=24.
\end{equation}
Based on $solvecount(c(2,3,4),24)$, we learn that there are 7 possible $[A]$ margins. Furthermore, there are $52937$ possible 3-way $[ABC]$ tables.
The linear Diophantine equation that  characterizes the missing marginal $[A]$ and thus $[AC]$ based on knowledge of $[C|A[$ is
\begin{equation}\label{eq:exc|a}
3x+3y+4z=24,
\end{equation}
and from the running $solvecount(c(3,3,4),24)$, we learn that there are 3 possible $A$ margins. There are $22440$ possible 3-way $ABC$ tables.

We are interested in the intersection of the two solution spaces. Using our function $intersect()$ in R, we learn that there is only one $[A]$ that satisfies both equations, and it takes values $(6, 6,12)$. Since there is only one $[A]$, this implies that there is only one $[AB]$ and one $[AC]$ margin, and thus the space of 3-way tables $[ABC]$ is the same as the space given these two margins. More specifically, $|\mathcal{F}_{\tau}|=|\mathcal{F}_{AB,AC}|=36$. 
Our analysis shows that the results from Section~\ref{subsec:dag} hold, and we do get the same bounds and Markov bases as would if we only consider the marginal information.  A Markov basis for fixed $[B|A]$ and $[C|A]$ has 5 elements: 3 fix the missing $[A]$ margin, and 2 change it:

\begin{footnotesize}
\[
 \left( {\begin{array}{*{20}c}
-4 &-2 &0 &-6 &0 &0 &0& 0& 3& 0& 0& 9\\
-2 &-1 &0 &-3 &2 &0 &0 &4 &0 &0 &0 &0 \\
0 &0 &0 &0 &-1 &1 &1 &-1 &0 &0 &0 &0 \\
0 &0 &0 &0 &0 &0 &0 &0 &-1 &1& 1 &-1 \\
-1 &1& 1 &-1 &0 &0& 0 &0 &0 &0 &0 &0 \\
\end{array}} \right).
\]
\end{footnotesize}

 Since in this example $[A]$ is unique, that would be like adding an additional constraint, and the actual minimal basis that describes our system of polynomial equations reduces to:
 
 \begin{footnotesize}
\[
 \left( {\begin{array}{*{20}c}
0 &0 &0& 0 &0 &0 &0 &0 &-1 &1 &1 &-1 \\
-1 &1 &1 &-1 &0 &0 &0 &0 &0 &0 &0 &0 \\
0 &0 &0 &0 &-1 &1 &1 &-1 &0 &0 &0 &0 \\
\end{array}} \right).
\]
\end{footnotesize}
We get the same Markov basis if we calculate it based on fixing $[AB]$ and $[AC]$ margins.

If $N=240,$ the Markov bases based on fixing $[B|A]$ and $[C|A]$ will be the same as with $N=24$; that is, they will have 5 elements shown above. However, now there are 361 possible $[A]$ margins consistent with both $[B|A]$ and $[C|A]$, and the Theorem~\ref{th:DAGmargin} and Corollary~\ref{th:DAGMbasis} and  are not satisfied, and the Markov basis will not reduce to the Markov basis given the  corresponding marginals. Furthermore, the space of tables given the conditional is significantly larger than the space of tables given the corresponding marginals: $|\mathcal{F}_{t}|=3066315\geq|\mathcal{F}_{AB,AC}|=13671$. Thus, the bounds on the cell entries are different, as is the support for the sampling distribution over the space of tables $[ABC]$. 

Similar analysis can be done for other arbitrary collections of conditionals and marginals. For example, $\mathcal{T}=\{P(B|A),P(A|C),P(C)\}$ will also satisfy the results from Section~\ref{subsec:dag}. If margin $[C]$ is missing, but it is unique based on the  solution to a linear Diophantine equation, we would again have a reduction of results; that is, the space $\mathcal{F}_t$  will be equivalent to the space $\mathcal{F}_{AB, AC}.$  For additional examples, see \url{http://www.stat.psu.edu/~sesa/cctable}.

\section{Conclusions}
\label{sec:conclusions}

We have used algebraic statistics to solve an open problem posed by \cite{dobra2008asa}. One area of this expanding field is concerned with 
the study and characterization of portions of the sample space and, in particular, of all datasets (i.e., tables) having the same observed margins and/or conditionals. 
In this paper,  we describe the space of all possible $r$-way contingency tables for a given sample size and set of observed (estimated) conditional frequencies.
This space of contingency tables can be decomposed according to different possible marginals, which, in turn, are encoded by the solution set to a linear Diophantine equation, giving the table space a special structure.  As a consequence, we obtain conditions under which two spaces of tables coincide: one is the space of tables for a given set of marginals, and the other is our space-- for a given sample size and set of conditionals. This characterization of the difference between two fibers has  thus provided a solution to an open problem in the literature.

In general, these fibers can be quite large. We provide formulas for computing the approximate and exact cardinality of the fibers in question, and we implemented those in R. The knowledge of the structure of the space of tables also enables us to enumerate all the possible data tables. This, in turn, leads to new cell bounds, some including connections to DAGs with combinations of conditionals and marginals.  In this paper, we assumed that the given sets of conditionals and marginals are compatible; for problems on compatibility for categorical and continuous variables, see \cite{arnold1999}; on compatibility of  full conditionals for discrete random variables, see \cite{slav:sull:2006}; and on generalization of compatibility of conditional probabilities in discrete cases, see \cite{morton2008}. Consistent with the literature on the characterizations of joint discrete distributions, we allow cell entries to be zero as long as we do not condition on an event of zero probability, and we assumed that the uniqueness theorems as stated in \cite{arnold1999} and \cite{slav:fien:2008}) hold. Then we considered if the given summary statistics are sufficient to uniquely identify the existing joint distribution, and if not, we proceed with the description of the related sample space.

Another application of the main observation, the table-space decomposition result, is that it allows us to describe the Markov bases given the conditionals. We observe that the moves consist of two sets: those that fix the margins, and those that change them. This result could lead to a simplified calculation of Markov bases in this particular setting. However, this remains to be studied more carefully. We raised a number of conjectures, and in particular we hope to prove Conjecture \ref{prop:conjecture}.

The properties of fibers, and, therefore, the results of this paper, are important in determining the support of sampling distributions, for conducting exact conditional inference, calculating cell bounds in contingency tables, and imputing missing cells in tables. The degree of Markov moves for given conditionals is arbitrary in the sense that it depends on the values of observed conditional probabilities, unless we use the observed cell counts directly. In practice, however, the conditional values are reported as real numbers.  Depending on the rounding point, the bounds, the moves and the fibers will differ from each of its kind. This has implications for statistical inference; in particular, in assessing ``true" disclosure risk in data privacy problems. The effect of rounding needs more careful investigation. This problem is related to characterizing when the integral approximation of the number of tables is correct up to rounding, and when the error is ''too large."

\appendix
\section{Proofs}
\label{sec:app-proofs}

\subsection{Proof of Theorem \ref{thm:diophantineSolutions=Marginals}}
\begin{proof}
    Assume 
    $\bf{n}$ is a table consistent with the given conditional $\{c_{ij}\}$ and grand total $N$.
    We can summarize the table using 
    $\bf{n^*}$ as described in the Introduction. 
    Thus  $\ds\frac{g_{ij}}{h_{ij}}=\ds\frac{s_{ij+}}{s_{+j+}}$.
    Since $g_{ij}$ and $h_{ij}$ are relatively prime, 
    it follows that 
    $s_{+j+}$ is an integer multiple of $h_{ij}$.  Furthermore, this is true for any $i$.
    By definition of $m_j$, $s_{+j+}$ is an integer multiple of $m_j$.
    In other words, 
    we can write $s_{+j+}$ as $m_{j}\cdot x_{j}$ where 
    where $x_j$ is a positive integer.  Now  Equation (\ref{eq:diophantine}) is satisfied since by definition 
    $\sum\limits_{j}{s_{+j+}}=N$.
    Conversely, assume (\ref{eq:diophantine}) holds for the positive
    integers $x_{j}$'s. Then we construct $\bf{n}$ by letting
    $s_{ij+}$ to be $m_{j}\cdot x_{j}\cdot c_{ij}$. Then let $s_{ijk}$
    to be nonnegative integers according to the equation
    $s_{ij+}=\sum\limits_{k}{s_{ijk}}$. Then construct $\bf{n}$
    according to $\bf{n^*}$ in a similar way.
\end{proof}

\subsection{Proof of Lemma \ref{lm:structureOfSolnsOfDioph}}
\label{sec:proof-of-thm:structure}
In the following, let $(m_1,\dots,m_l)$ denote the greatest common divisor of $m_1,\dots,m_l$ for any arbitrary $1\leq l \leq J$. Notice that the standard Euclidean algorithm produces  integers $x_1^0,\dots,x_J^0$ such that $m_1x_1+\dots+m_lx_l=(m_1,\dots,m_J)$. 
Repeatedly using this process, we get $x_i$'s such that $\sum\limits_{i=1}^{J-j}{m_i x_i^{(j)}}=(m_1,m_2,...,m_{J-j})$ for any $j$. 
In particular, we can set 
$x_i=x_i^{(0)}\cdot\frac{N}{(m_1,...,m_J)}$ to obtain one of the  integer solutions of (\ref{eq:diophantine}). Note that this algorithm performs at most  $\sum\limits_{i=1}^{J}{m_i}$ calculations. 
Similarly, every solution of the Diophantine equation can be obtained by integers linear combinations, generalizing the two basic examples. 

\begin{proof}
    Elementary arguments allow us to express the vectors ${\bf v}_1,\dots,{\bf v_{J-1}}$    in terms of the coefficients $m_1$, $\dots$, $m_J$.
    By the Euclidean algorithm,  the \emph{gcd}'s $(m_1,\dots,m_l)$ can be expressed as a linear combination of the $m_j$'s: 
    \begin{align*}
        \sum\limits_{i=1}^{J-j}{m_ix_i^{(j)}}=(m_1,m_2,...,m_{k-j})
    \end{align*}
    for $j=1,2,...,J$.
    Then we can express all integer solutions of Equation (\ref{eq:diophantine}) as:
    \begin{align*}
    x_l=&x_l^{(0)}-\sum\limits_{h=1}^{J-l}{\ds\frac{m_{J+1-h}x_{l}^{(h)}}{(m_1,...,m_{J+1-h})}\cdot q_h}+
        \ds\frac{(m_1,...,m_{l-1})} {(m_1,...,m_l)}\cdot q_{J-l+1} \mbox{ for } l=2,\dots,J, \\
    x_1=&x_1^{(0)}-\sum\limits_{h=1}^{J-1}{\ds\frac{m_{J+1-h}x_{1}^{(h)}}{(m_1,...,m_{J+1-h})}\cdot q_h},
    \end{align*}
    where $q_i\in\mathbb Z$ for all $i$ with $1\leq i\leq J-1$.
    Then the vectors ${\bf v}_i$, for $i=1,\dots,J-1$, are determined from these expressions as follows:
    the $l^{th}$ coordinate of $v_i$ is the coefficient of $q_i$ in the expression for $x_l$.
\end{proof}

\subsection{Proof of Proposition \ref{prop:approxNumOfSolnsDiophantine-viaAlgebra}}
\label{sec:proof-of-prop:approxNumOfSolnsDiophantine-viaAlgebra}
\begin{proof}
    To approximate the number of nonnegative solutions, define a vector $u:=[m_1,m_2,...m_k]^T$, and a matrix
    $A:=[u,v_1,v_2,...,v_{k-1}]$.  Recall that vectors $v_1, ..., v_k$ come from Lemma~\ref{lm:structureOfSolnsOfDioph}. From the expressions above, we see that
    \begin{align*}
        A=\left[ {\begin{array}{*{20}c}
          {m_1 } & {\frac{{m_k x_1 ^{(1)} }}{{(m_1 ,\dots,m_k )}}} & {\frac{{m_{k - 1} x_1 ^{(2)} }}{{(m_1 ,\dots,m_{k - 1} )}}} & {\dots} & {\dots} &
        {\frac{{m_2 x_1 ^{(k - 1)} }}{{(m_1 ,m_2 )}}}  \\
           {m_2 } & {\frac{{m_k x_2 ^{(1)} }}{{(m_1 ,\dots,m_k )}}} & {\frac{{m_{k - 1} x_2 ^{(2)} }}{{(m_1 ,\dots,m_{k - 1} )}}} & {\dots} & {\dots} &
        {\frac{{ - m_1 }}{{(m_1 ,m_2 )}}}  \\
           {m_3 } & {\frac{{m_k x_3 ^{(1)} }}{{(m_1 ,\dots,m_k )}}} & {\frac{{m_{k - 1} x_3 ^{(2)} }}{{(m_1 ,\dots,m_{k - 1} )}}} & {\dots} &
        {\frac{{ - (m_1 ,m_2 )}}{{(m_1 ,m_2 ,m_3 )}}} & 0  \\
          {\vdots} & {\vdots} & {\vdots} & {\vdots} & {\vdots} & {\vdots}  \\
           {m_{k - 1} } & {\frac{{m_k x_{k - 1} ^{(1)} }}{{(m_1 ,\dots,m_k )}}} & {\frac{{ - (m_1 ,\dots,m_{k - 2} )}}{{(m_1 ,\dots,m_{k - 1} )}}} & 0 &
        {\dots} & 0  \\
           {m_k } & {\frac{{ - (m_1 ,\dots,m_{k - 1} )}}{{(m_1 ,\dotsm_k )}}} & 0 & 0 & {\dots} & 0  \\
        \end{array}} \right]
    \end{align*}
    
    One readily checks that $u$ is orthogonal to any column $v_i$.  Thus the absolute value of $\ds(\det{A})/{||u||}$
        is the (k-1)-dimensional volume of the parallelotope spanned by $v_1,v_2,...,v_{k-1}$. Let's compute this value:
        
        \begin{scriptsize}
    \begin{align*}
    \frac{{\det A}}{{||u||}} &= \frac{1}{{\sqrt {m_1 ^2  + m_2 ^2  + ...
    + m_k ^2 } }} \cdot \det A \\
    &=\frac{1}{{\sqrt {m_1 ^2  +\dots + m_k ^2 } }}
    \cdot \det\left[ {\begin{array}{*{20}c}
       {m_1 } & {\frac{{m_k x_1 ^{(1)} }}{{(m_1 ,\dots,m_k )}}} & {\frac{{m_{k - 1} x_1 ^{(2)} }}{{(m_1 ,\dots,m_{k - 1} )}}} & {\dots} & {\dots} &
    {\frac{{m_2 x_1 ^{(k - 1)} }}{{(m_1 ,m_2 )}}}  \\
       {m_2 } & {\frac{{m_k x_2 ^{(1)} }}{{(m_1 ,\dots,m_k )}}} & {\frac{{m_{k - 1} x_2 ^{(2)} }}{{(m_1 ,\dots,m_{k - 1} )}}} & {\dots} & {\dots} &
    {\frac{{ - m_1 }}{{(m_1 ,m_2 )}}}  \\
       {m_3 } & {\frac{{m_k x_3 ^{(1)} }}{{(m_1 ,\dots,m_k )}}} & {\frac{{m_{k - 1} x_3 ^{(2)} }}{{(m_1 ,\dots,m_{k - 1} )}}} & {\dots} &
    {\frac{{ - (m_1 ,m_2 )}}{{(m_1 ,m_2 ,m_3 )}}} & 0  \\
       {\vdots} & {\vdots} & {\vdots} & {\vdots} & {\vdots} & {\vdots}  \\
       {m_{k - 1} } & {\frac{{m_k x_{k - 1} ^{(1)} }}{{(m_1 ,\dots,m_k )}}} & {\frac{{ - (m_1 ,\dots,m_{k - 2} )}}{{(m_1 ,\dots,m_{k - 1} )}}} & 0 &
    {\dots} & 0  \\
       {m_k } & {\frac{{ - (m_1 ,\dots,m_{k - 1} )}}{{(m_1 ,\dots,m_k )}}} & 0 & 0 & {\dots} & 0  \\
    \end{array}} \right] \\
        &=\frac{1}{m_1{\sqrt {m_1 ^2  +\dots + m_k ^2 } }}
        \cdot\det\left[ {\begin{array}{*{20}c}
           {\sum\limits_{i=1}^{k}{m_i^2} } & 0 & 0 & {\dots} & {\dots} & 0  \\
       {m_2 } & {\frac{{m_k x_2 ^{(1)} }}{{(m_1 ,\dots,m_k )}}} & {\frac{{m_{k - 1} x_2 ^{(2)} }}{{(m_1 ,\dots,m_{k - 1} )}}} & {\dots} & {\dots} &
    {\frac{{ - m_1 }}{{(m_1 ,m_2 )}}}  \\
       {m_3 } & {\frac{{m_k x_3 ^{(1)} }}{{(m_1 ,\dots,m_k )}}} & {\frac{{m_{k - 1} x_3 ^{(2)} }}{{(m_1 ,\dots,m_{k - 1} )}}} & {\dots} &
    {\frac{{ - (m_1 ,m_2 )}}{{(m_1 ,m_2 ,m_3 )}}} & 0  \\
       {\vdots} & {\vdots} & {\vdots} & {\vdots} & {\vdots} & {\vdots}  \\
       {m_{k - 1} } & {\frac{{m_k x_{k - 1} ^{(1)} }}{{(m_1 ,\dots,m_k )}}} & {\frac{{ - (m_1 ,\dots,m_{k - 2} )}}{{(m_1 ,\dots,m_{k - 1} )}}} & 0 &
    {\dots} & 0  \\
       {m_k } & {\frac{{ - (m_1 ,\dots,m_{k - 1} )}}{{(m_1 ,\dots,m_k )}}} & 0 & 0 & {\dots} & 0  \\
    \end{array}} \right] \\
    &=\frac{{( - 1)^{k - 1} \sqrt {m_1 ^2   + \dots + m_k ^2 }
    }}{{(m_1,m_2, \dots ,m_k )}} .
    \end{align*}
    \end{scriptsize}
    
    Thus the volume of the parallelotope spanned by $v_1,v_2,...,v_{k-1}$ is  
    \[
        \frac{{ \sqrt {m_1 ^2  + m_2 ^2  + ... + m_k ^2 } }}{{(m_1,m_2, ...,m_k )}}.
    \]

    Next, define
    $$G=\{(x_1,...,x_k)^T|m_1x_1+m_2x_2...+m_kx_k=N,x_1\geq0,x_2\geq0,...,x_k\geq0\}.$$
    LetÕs refer to $G$ as the \emph{marginal polytope}. The volume of $G$ is easily calculated to be    \[
        \frac{N^{k-1}}{(k-1)!(m_1\cdot m_2\cdot \dots \cdot m_k)}\sqrt{m_1^2+m_2^2+...+m_k^2}
    \]

    The approximation to the number of lattice points in $G$, that is, the number of positive integer solutions of (\ref{eq:diophantine})
    is obtained by dividing the volume of $G$ by the volume of the parallelotope above.
 This proves the first claim.
   
   For the second claim, let    $\mathcal{P}$ be the projection of the set of positive integer solutions onto
    the $x_1\dots x_{j-1}$ -plane. Then there are exactly 
    \[
        \sum_{x_1x_2\dots x_{j-1}\in\mathcal{P}}1
    \]
 positive integer solutions of the Diophantine equation (\ref{eq:diophantine}). 
  Let $\mathfrak a$ be the area of the unit cell of the lattice spanned by $\mathcal P$.
 Then
  \[
 \int\limits_{(x_1,...,x_{j- 1} ) \in \mathcal{M}} 1dx_1\dots dx_j  \approx   \mathfrak a  \cdot \sum_{x_1x_2\dots x_{j-1}\in\mathcal{P}}1,
  \]
where the right hand side is, by definition, the Riemann sum approximation of the integral. 
In particular, one easily concludes that the error of this approximation is given by the difference in the volume of the polytope $\mathcal M$ and the volume of the polyhedron which is the union of all the unit cells anchored at the lattice points $\mathcal P$. 

To complete the proof, we calculate the area of the unit cell $\mathfrak a$. 
Let   $\mathcal L$ be the lattice of all integer solutions to Equation (\ref{eq:diophantine}).    Since $\mathcal P\subseteq \mathcal L\cup \{x_j=0\}$, 
    we can choose its unit cell to be the projection of the unit cell of $\mathcal L$ onto $\{x_j=0\}$.
This projection, in turn, 
is a
    parallelopiped whose $(j-1)$-dimensional
    volume is the absolute value of
    $$
    \det\left[ {\begin{array}{*{20}c}
       {\frac{{m_j x_2 ^{(1)} }}{{gcd(m_1 ,...,m_j )}}} & {\frac{{m_{j - 1} x_2 ^{(2)} }}{{gcd(m_1 ,...,m_{j - 1} )}}} & {...} & {...} & {\frac{{ -
    m_1 }}{{gcd(m_1 ,m_2 )}}}  \\
      {\frac{{m_j x_3 ^{(1)} }}{{gcd(m_1 ,...,m_j )}}} & {\frac{{m_{j - 1} x_3 ^{(2)} }}{{gcd(m_1 ,...,m_{j - 1} )}}} & {...} & {\frac{{ -
    gcd(m_1 ,m_2 )}}{{gcd(m_1 ,m_2 ,m_3 )}}} & 0  \\
      {...} & {...} & {...} & {...} & {...}  \\
        {\frac{{m_j x_{k - 1} ^{(1)} }}{{gcd(m_1 ,...,m_j )}}} & {\frac{{ - gcd(m_1 ,...,m_{j - 2} )}}{{gcd(m_1 ,...,m_{j- 1} )}}} & 0 & {...} & 0  \\
      {\frac{{ - gcd(m_1 ,...,m_{j - 1} )}}{{gcd(m_1 ,...,m_j )}}} & 0 & 0 & {...} & 0  \\
    \end{array}} \right]
    $$
    which is $\ds\frac{m_1}{gcd(m_1,m_2,...,m_j)}$.
\end{proof}

\section{Code \& Examples}
\label{sec:app-code}

\begin{example}
    Let us consider a bivariate ($J=2$) Diophantine equation
    \begin{align}
    \label{eq:dio-xy}
    ax + by =N,
    \end{align}
    where $a:=m_1$, $b:=m_2$, and $N$ are positive integers. Note that we have renamed the variables $x:=x_1$ and $y:=x_2$
    for simplicity of notation.

Let $L$ be  the line defined by \eqref{eq:dio-xy} for $(x,y)\in\mathbb R^2$.
  We are only interested in the set of \emph{nonnegative integer} solutions to (\ref{eq:dio-xy}), that is, \emph{nonnegative lattice points} $L\cap \mathbb Z^2_{\geq 0}$ on the line $L$. 
    Every ideal in $\mathbb Z$ can be generated by one element; in our case, this element is the
    \emph{greatest common divisor} of $a$ and $b$, which we will denote by $gcd(a,b)$.
    In particular, it follows that the equation (\ref{eq:dio-xy}) has integer solutions \emph{if and only if} $gcd(a,b)$ divides $N$.
    In addition, the description of \emph{all} integral solutions readily follows by elementary algebra.
    Namely, suppose that $(x_0,y_0)\in\mathbb Z^2$ is one integer solution of $ax +by =N$.  Then all other integer solutions are given by the following equation where $q$ is an arbitrary integer:
    \begin{align}
    \label{eq:dio-xy-solnformula}
    \left\{ \begin{array}{l}
    x = x_0  + \frac{b}{{gcd(a,b)}} \cdot q \\
    y = y_0  - \frac{a}{{gcd(a,b)}} \cdot q  \\
    \end{array}  \right.
   \end{align}

    In fact, we can also estimate the {number of solutions} of (\ref{eq:dio-xy}).
    The geometry of the line provides that $x\in [0,N/a]$.  From (\ref{eq:dio-xy-solnformula}), it follows that $x$ varies by multiples of $b/gcd(a,b)$.
    Therefore, there are at most
    \begin{align*}
        \frac{N/a}{b/gcd(a,b)}=\frac{N \cdot gcd(a,b)}{ab}
    \end{align*}
    points in $L \cap \mathbb Z^2_{\geq 0}$.  Note that this is only an estimate, albeit a good one, since we are essentially counting only
    $\{x : ax+by=N \mbox{ for some $y$}\} \cap \mathbb Z$.
\end{example}

The code used for the analysis in this paper and additional examples are available at \url{http://www.stat.psu.edu/~sesa/cctable}

The examples suggest that, in general, we are interested in the lattice points of the polytope obtained by intersecting the hyperplane defined by Equation \eqref{eq:diophantine} and the positive orthant.

The code includes:
\begin{itemize}
\item A collection of functions we wrote in R for enumerating and counting the number of missing marginal and $k$-way tables given the partial information described in the paper. There are functions for  (1) finding the greatest common divisor, (2) solving Diophantine equations, and (3) counting the number of tables. 
\item A sample R and LattE code for the examples in this paper, and some additional related examples. 
\item An additional example in support of Lemma~\ref{lm:structureOfSolnsOfDioph}. 
\end{itemize}

\bibstyle{abbrvnat}
\bibliography{algstat}

\end{document}